\documentclass{amsart}
\usepackage[utf8]{inputenc}
\usepackage{amsmath,times,hyperref, multicol}
\usepackage{amsthm}
\usepackage{amssymb}
\usepackage{amsfonts}
\usepackage{latexsym}
\usepackage{graphicx}
\usepackage[usenames,dvipsnames]{color} 
\usepackage{enumitem}
\setlist[itemize]{leftmargin=*}
\setlist[enumerate]{leftmargin=*}
\usepackage[font={small}]{caption}
\usepackage{subcaption}
\usepackage{cite}
\usepackage{mathtools}
\mathtoolsset{showonlyrefs}

\usepackage{tikz,amsmath,amsfonts}
\usetikzlibrary{calc,intersections}
\usetikzlibrary{decorations.pathreplacing}
\usepackage{ifthen} 
 \newcommand{\Nbar}{{\overline{\mathbb{N}}}}
\newcommand{\R}{{\mathbb{R}}}
\newcommand{\N}{{\mathbb{N}}}
\newcommand{\C}{{\mathbb{C}}}
\newcommand{\A}{{\mathcal{A}}}

\newcommand{\no}{\mathring{n}}

\newtheorem{theorem}{Theorem}[section]
\newtheorem{proposition}[theorem]{Proposition}

\theoremstyle{definition}
\newtheorem{definition}[theorem]{Definition}

\newtheorem{convention}[theorem]{Convention}

\theoremstyle{remark}

\newtheorem{example}[theorem]{Example}

\DeclareMathOperator{\Cl}{Cl}
\DeclareMathOperator{\Ideals}{Ideals}
\allowdisplaybreaks

\numberwithin{equation}{section}

\title{The Hamming Distance and the Fell Topology  on AF Algebras}

\author{Konrad Aguilar (Pomona College) and Zo\"{e} Batterman (Pomona College)}

\address{Department of Mathematics and Statistics, Pomona College, Claremont, CA 91711}
\email{konrad.aguilar@pomona.edu}
\urladdr{https://aguilar.sites.pomona.edu/}

\address{Department of Mathematics and Statistics, Pomona College, Claremont, CA 91711}
\email{zxba2020@mymail.pomona.edu}

\thanks{The first author is supported by NSF grant DMS-2316892}
\date{\today}

\begin{document}

  \maketitle
\begin{abstract} 
We introduce a new metric on the ideal space of an AF algebra that metrizes the Fell topology. The novelty of this metric lies in the use of a Hamming distance type metric in its construction. Furthermore,   this metric captures more of the ideal structure of AF algebras in comparison to known metrics on the Fell topology of an AF algebra. We explicitly test this on the C$^*$-algebra of complex-valued continuous functions on a quantized interval by comparing our new metric with the dual Hausdorff distance on the ideals of this C$^*$-algebras induced by the Hausdorff distance on the closed subsets of the quantized interval.
\end{abstract}
\tableofcontents
\section{Introduction and Background} 

Noncommutative metric geometry was introduced by M.A. Rieffel to provide a mathematical framework to   certain statements arising in the high energy physics literature  about continuous families of  operator algebras \cite{Rieffel98a, Rieffel00, Rieffel01}. The main objects Rieffel developed in noncommutative metric geometry are:  noncommutative analogues of compact metric spaces called compact quantum metric spaces and a noncommutative analogue to the Gromov-Hausdorff distance called the quantum Gromov-Hausdorff distance. Many more continuity results have been established with these new tools \cite{Junge18, Kaad20, Kyed22, Latremoliere-Packer17} (to only name a few) as well as various new noncommutative analogues of the quantum Gromov-Hausdorff distance \cite{Kerr02, li03, Latremoliere13b, wu06a}. Moreover, F. Latr\'emoli\`ere  recently built a series of noncommutative analogues of the Gromov-Hausdorff distance to capture  more structure associated to C$^*$-algebras such as module structure \cite{Latremoliere16} and spectral triple structure \cite{Latremoliere22a}. 

  In 1961 \cite{Fell61}, J.M.G. Fell introduced the Fell topology on the ideal space of C$^*$-algebras and provided the foundation for the theory of continuous fields of operator algebras and C$^*$-algebras. In \cite{Aguilar16c},  the first author  found a connection between noncommutative metric geometry and the Fell topology. In particular, it was found that for certain families of AF algebras, one can find a continuous map from the induced Fell topology to the quotients (viewed as  compact quantum metric spaces) equipped with the topology induced by Latr\'emoli\`ere's noncommutative analogue to the Gromov-Hausdorff distance, the Gromov-Hausdorff propinquity \cite{Latremoliere13b}. A metric that metrized the Fell topology was introduced in \cite{Aguilar16c} to prove this result.

Even though the metric of \cite{Aguilar16c} proposed answers the original question, it does not differentiate between varied ideal structures since it only tells us when two ideals first disagree in a given inductive sequence of an AF algebra and disregards the other rich structure found by O. Bratteli \cite{Bratteli72}.

In Section \ref{s:new-metric}, we build a metric which metrizes the Fell topology and that considers more structure of the ideals by using a Hamming distance type metric at each level of the inductive sequence. We explicitly calculate the ability of our metric to that of \cite{Aguilar16c} in an example in Section~\ref{sct:section-example}. The example we consider is the space of complex-valued continuous functions on a quantized interval to allow us to compare with a classical metric, the Hausdorff distance. In fact, for any unital separable commutative C$^*$-algebra, one obtains another metric on the Fell topology by way of the Hausdorff distance on the spectrum, a dual metric on the ideals induced by the spectrum. 
In the case of AF algebras, this dual metric gives us a classical metric which allows us to compare the metric of \cite{Aguilar16c} and the new metric introduced in this article. By focusing on the quantized interval, we find certain ideals which highlight the fact our new metric sees more structure since it behaves more like the dual Hausdorff distance than the metric of \cite{Aguilar16c}.

We continue this section with the necessary background for the rest of the article.  

\begin{convention}
Let $\A$ be a C$^*$-algebra. By an ideal of $\A$, we mean a norm-closed two-sided algebraic ideal. We denote the set of ideals of $\A$ by $\Ideals(\A)$. We say $\A$ is simple if $\Ideals(\A)=\{\{0\}, \A\}$.
\end{convention}

The main example of this paper comes from the following class of C$^*$-algebras, where we also include a characterization of the ideals for an element in this class for   later use. 
\begin{example}\label{e:comm}
Let $X$ be a compact Hausdorff space. Let 
\[
C(X)\coloneqq \{f:X \rightarrow \C \mid f \text{ is continuous}\}
\]
denote the C$^*$-algebra of complex-valued continuous functions on $X$ with respect to pointwise operations; that is, the adjoint is defined by  $f^*(x)=\overline{f(x)}$ for every $f\in C(X)$ and every $x \in X$.
We denote the supremum norm for every $f \in C(X)$ by
\[
\|f\|_\infty\coloneqq \sup_{x \in X} |f(x)|.
\] 

Moreover, the ideals of $C(X)$ are in one-to-one correspondence with the closed subsets of $X$. Indeed, let $\Cl(X)$ denote the set of closed subsets of $X$. This one-to-one correspondence is given by the map
\[
\Phi_I: A \in \Cl( X) \longmapsto I(A) \in \Ideals(C(X))
\]
where
\[
I(A)=\{ f \in C(X) : f(a)=0 \text{ for all }a \in A \}.
\]
(see \cite[Chapter VII.8]{Conway90}).
\end{example}

Bratteli \cite{Bratteli72} introduced the notion of an approximately finite-dimensional C$^*$-algebra at the same time as introducing certain infinite graphs called Bratteli diagrams.  We note these diagrams were vital to our understanding of the ideal space and to the construction of our metric.

\begin{definition}[{\cite{Bratteli72}}]\label{d:af}
Let $\A$ be a unital C$^*$-algebra. We say $\A$ is an {\em AF algebra} if for every $n \in \N$, there exists a finite-dimensional unital C$^*$-subalgebra $\A_n$ of $\A$ such that
    \begin{enumerate}
        \item $\A_n \subseteq \A_{n+1}$,
        \item $\A=\overline{\cup_{n \in \N}\A_n}^{\|\cdot\|_\A}$.
    \end{enumerate}

    For each $n \in \N$, we call $\A_n$ the {\em $n$-level} of $\A=\overline{\cup_{n \in \N}\A_n}^{\|\cdot\|_\A}$, and we call $(\A_n)_{n \in \N}$ the {\em inductive sequence}.
\end{definition}

 The main example of this paper is the commutative C$^*$-algebra of complex-valued continuous functions on the quantized interval. Consider the quantized interval
\[
\Nbar =\left\{ \frac{1}{2^{n-1}} \in \R : n \in \N\right\} \cup \{0\}.
\]
Then $C\left( \Nbar\right)$
is an AF algebra since
\[
C\left( \Nbar\right)=\overline{\cup_{n \in \N} C_n}^{\|\cdot\|_\infty},
\]
where for each $n \in \N$, $C_n$ is the unital C$^*$-subalgebra of functions that are eventually constant at $2^{1-n}$, that is
\[
C_n=\left\{ f\in C(\Nbar) : \forall x \in \Nbar, x \leq 2^{1-n}, f(x)=f(2^{1-n})\right\},
\]
  which is finite-dimensional of dimension $n$ since it is the span of the characteristic functions
\[
\chi_{\{1\}}, \chi_{\{1/2\}}, \ldots, \chi_{\{2^{2-n}\}}, \chi_{[0,2^{1-n} ]}.
\]

Next, the first author \cite{Aguilar16c} introduced a metric on the ideal space of AF algebras that metrizes the Fell topology. In this paper, we see  this metric doesn't capture much information about the ideal structure. In particular, we see this happen when we explicitly compare it on $C(\Nbar)$ with another standard metric that metrizes the Fell topology induced by the Hausdorff distance. We will wait until the next section to define this Hausdorff distance  induced     metric, and for now, we introduce the metric in \cite{Aguilar16c}. But first, we define the Fell topology which begins with the Jacobson topology.

\begin{definition}\label{Jacobson-topology}Let $\A$ be a C$^*$-algebra.   Denote    the set 
\begin{equation*}
\begin{split}
 \left\{ \ker \pi \in \mathrm{Ideals}(\A) :   \pi \text{ is an irreducible $^*$-representation of }   \A \text{ with } \pi \neq 0 \right\}
\end{split}
\end{equation*}
of {\em Primitive ideals of $\A$} by  $ \mathrm{Prim}(\A)$.
Note that $\A \not\in \mathrm{Prim}(\A)$.

The {\em Jacobson topology} is the topology on $\mathrm{Prim}(\A)$, denoted by  $Jacobson$,  such that a set $F$ is closed if there exists a unique ideal $I_F \in \mathrm{Ideals}(\A)$ such that $F=\{ J \in \mathrm{Prim}(\A) : J \supseteq I_F\}$  (see Theorem 5.4.2, Theorem 5.4.6,  and Theorem 5.4.7 in \cite{Murphy90}).
\end{definition}
Now, we may define the Fell topology on $\Ideals(\A)$, which is a topology on all ideals of a C$^*$-algebra.  However, to build this, Fell first defined a topology on closed sets of any topological space along, which we define now.
\begin{definition}[\cite{Fell62}]\label{topological-Fell-def}
Let $(X, \tau)$ be a topological space  (no further assumptions made).   Let $K$ be a compact set of $X$, and let $F$ be a finite family of non-empty open subsets of $X$. Define:
\begin{equation*}
U(K,F)=\{  Y \in \Cl(X) : Y \cap K = \emptyset  \text{ and }  Y\cap A\neq \emptyset \text{ for all } A \in F\}.
\end{equation*}  The Fell topology $\tau_{ \Cl(X)}$ on $\Cl(X)$  is generated by the topological basis:
\begin{equation*}
\left\{ U(K,F) \subseteq \Cl(X): K\subseteq X \text{ is compact} \text{ and }  F\subseteq \tau\setminus \{\emptyset\}, F \ \text{ finite}\right\}. 
\end{equation*}
\end{definition}
 
\begin{definition}[{\cite{Fell61}}]\label{Fell-topology}
Let $\A$ be a C$^*$-algebra.  Define
 \[fell : I \in  \mathrm{Ideals}(\A) \longmapsto \left\{ J \in \mathrm{Prim}(\A) : J \supseteq I\right\},\] 
which is a bijection onto  $ \Cl(\mathrm{Prim}(\A))$ by Theorem 5.4.7 in \cite{Murphy90}.

The {\em Fell topology} on $\mathrm{Ideals}(\A)$, denoted $Fell$, is the initial topology on $\mathrm{Ideals}(\A)$ induced by $fell$  and  the Fell topology on $\Cl(\mathrm{Prim}(\A))$ of Definition \ref{topological-Fell-def} induced  by  the Jacobson topology on $\mathrm{Prim}(\A)$. 

In particular,  this topology is given by: 
\begin{equation*}
 Fell=\big\{ U \subseteq \mathrm{Ideals}(\A) : U=fell^{-1} (V), V \in \tau_{\Cl(\mathrm{Prim}(\A))}\big\}.
\end{equation*}
\end{definition}

Next, we present the metric from \cite{Aguilar16c} that metrizes the Fell topology on AF algebras.
\begin{theorem}[{\cite[Theorem 3.22]{Aguilar16c}}]\label{t:a-metric}
Let $\A=\overline{\cup_{n \in \N}\A_n}^{\|\cdot\|_\A}$ be an AF algebra, where we set $\mathcal{I}=(\A_n)_{n \in \N}$. The function $d_{\varphi, \mathcal{I}}: \Ideals(\A) \times \Ideals (A) \rightarrow [0, \infty)$ defined by 
\[
d_{\varphi, \mathcal{I}}(I,J)=\begin{cases}
0 & I=J,\\
    2^{-\min \{m \in \N: I\cap \A_m \neq J \cap \A_m\}} & I \neq J
\end{cases}
\]
for every $I,J \in \Ideals (\A)$ is a metric on $\Ideals(\A)$ that metrizes the Fell topology.
\end{theorem}

Thus, our main goal is to introduce a metric that sees more of the ideal structure of an AF algebra. To accomplish this, we need a little more information about the structure of AF algebras. By \cite[Theorem III.1.1]{Davidson}, we have that any finite-dimensional C$^*$-algebra is  a finite direct sum of simple C$^*$-subalgebras. In fact, any simple finite-dimensional C$^*$-algebra is $^*$-isomorphic to the algebra of $n\times n$-matrices with complex entries, $M_n(\C)$. Therefore given an AF algebra
\[
\A=\overline{\cup_{n \in \N}\A_n}^{\|\cdot\|_\A},
\]
for each $n \in \N$, there exists $ n_n\in \N$ such that 
\[
\A_n=\bigoplus_{k=1}^{n_n} \A_{n,k},
\]
where $\A_{n,k}$ is a simple C$^*$-subalgebra of $\A_n$ for each $k \in \{1,2,\ldots, n_n\}$.  Our new metric will be built with this decomposition of each $\A_n$ into simple summands, whereas  the metric of \cite{Aguilar16c} does not utilize this in any way in its   construction. 
Relating back to our main example, for each $n \in \N$, we have
\begin{equation}\label{eq:nbar}
    C_n=  
\bigg(\bigoplus_{k=1}^{n-1} \C \chi_{\{2^{-[1-k]}\}} \bigg) \oplus \C\chi_{[0,2^{1-n}]},
\end{equation}
where $\C f=\{\lambda f \in C(\Nbar): \lambda \in \C\}$ for every $f \in C(\Nbar)$. We set 
\[C_{n,k}=\C \chi_{\{2^{-[1-k]}\}},
\]
for $k=1, 2, \ldots, n-1$ and set 
\[C_{n,n} = \C\chi_{[0,2^{1-n}]},
\]
and we use this ordering of the summands for the rest of the paper.

With this detail, we can describe an important property of the ideals of an AF algebra that will allow us to build this a more refined metric.
\begin{theorem}[{\cite[First sentence of proof of Theorem III.4.2]{Davidson}}]\label{t:ideal-af}
    Let $\A=\overline{\cup_{n \in \N}\A_n}^{\|\cdot\|_\A}$ be an AF algebra. If $I \in \Ideals(\A)$, then for each $n \in \N$, we have 
    \[
    I\cap \A_n=\bigoplus_{k \in \no(I)} \A_{n,k},
    \]
    where $\no(I) \subseteq \{1,2,\ldots, n_n\},$ and we order $\no(I)=\{\no(I,1), \no(I,2), \ldots, \no(I,n_I)\}$, where $n_I\leq n.$
\end{theorem}

\section{A new metric on the ideal space of AF algebras}\label{s:new-metric}

Building immediately from Theorem \ref{t:ideal-af}, we may define our new metric. The idea behind the new metric is that it captures how distinct the ideals are at every level of the AF algebra, whereas the metric of $d_{\varphi, \mathcal{I}}$ of Theorem \ref{t:a-metric} only looks at the first level where the ideals disagree.  Our approach uses part of the construction of the Hamming distance at each level. Indeed, given a vector (in $\R^n$, for example), the Hamming distance detects every coordinate in which the vectors disagree and then takes the cardinality of the number of coordinates that disagree \cite{Hamming50}. In our construction, we  view the ideals at each level as vectors motivated by Bratteli diagrams of ideals \cite[Theorem 3.3]{Bratteli72} and then detect where they disagree. Then we combine this information in the following way to form a metric  on  of the set of ideals of an AF algebra.  

\begin{definition}
    Let $\A=\overline{\cup_{n \in \N}\A_n}^{\|\cdot\|_\A}$ be an AF algebra, where we set $\mathcal{I}=(\A_n)_{n \in \N}$. For every $I,J \in \Ideals (\A)$ and every $n \in \N$, we set
    \[
    \no(I,J)=\no(I)\triangle\no(J),
    \]
    where $\triangle$ denotes symmetric difference. 
    We define for every $I,J \in \Ideals (\A)$
    \[
    d_{\beta, \mathcal{I}}(I,J)\coloneqq \sum_{n=1}^\infty \sum_{k \in \no(I,J)} 2^{-[n+k]},
    \]
    with the convention that $\sum_{k \in \emptyset}(\cdot ) =0.$
\end{definition}

\begin{theorem}
     Let $\A=\overline{\cup_{n \in \N}\A_n}^{\|\cdot\|_\A}$ be an AF algebra, where we set $\mathcal{I}=(\A_n)_{n \in \N}$. The function $ d_{\beta, \mathcal{I}}$ is a metric on $\Ideals(\A)$.
\end{theorem}
\begin{proof}
 First, we check this sum is finite for every $I,J \in \Ideals (\A)$. For any $n \in \N$, observe $\no(I,J)\leq n$ by definition of symmetric difference as $\no(I), \no(J)\leq n$. Thus
\begin{equation}\label{eq:sum}
\begin{split}
d_{\beta,\mathcal{I}} (I,J) & = \sum_{n=1}^{\infty} \sum_{k \in \mathring{n}(I,J)} 2^{-[n+k]} \leq \sum_{n=1}^{\infty} \sum_{k=1}^{n} 2^{-[n+k]}    =\sum_{n=1}^{\infty} 2^{-n}  \sum_{k=1}^{n} 2^{-k} \\
&= \sum_{n=1}^{\infty} 2^{-n} \bigg( \frac{1}{2} \bigg( \frac{1 - (\frac{1}{2})^{n}}{1-\frac{1}{2}} \bigg)  \bigg) = \sum_{n=1}^{\infty} 2^{-n} - \sum_{n=1}^{\infty} 2^{-2n} = \frac{2}{3} 
.\end{split}
\end{equation}
The axiom of coincidence follows since $I = J$ if and only if $\no(I,J) = \emptyset$ for all $n\in \mathbb{N}$ by \cite[Lemma III.4.1]{Davidson} if and only if $\sum_{k \in \emptyset} 2^{-[n+k]} = 0$ for all $n \in \N$ if and only if $d_{\beta,\mathcal{I}} (I,J) = 0$.
Symmetry follows since $\mathring{n}(I,J) = \mathring{n}(I) \, \triangle \, \mathring{n}(J) = \mathring{n}(J) \, \triangle \, \mathring{n}(I) = \mathring{n}(J,I)$.
Observe the triangle inequality holds since if $K \in \Ideals(\A)$, then 
\begin{align*}
\no(I,J) = \no(I) \, \triangle \, \no(J) \subseteq ( \no(I) \, \triangle \, \no(K)) \cup  (\no(K)\, \triangle\, \no(J) ) = \no(I,K) \cup \no(K,J)	
\end{align*}
by symmetric difference, 
and so
\begin{align*}
\sum_{k \in \mathring{n}(I,J)} 2^{-[n+k]} & \leq \sum_{k \in \mathring{n} (I,K) \cup \mathring{n} (K,J)} 2^{-[n+k]} \leq \sum_{k \in \mathring{n}(I,K)} 2^{-[n+k]} + \sum_{k \in \mathring{n}(J,K)} 2^{-[n+k]}
.\qedhere\end{align*} 
 \end{proof}

We end this section by providing a comparison with $ d_{\varphi, \mathcal{I}}$ that in fact shows our new metric $d_{\beta, \mathcal{I}}$ also metrizes the Fell topology. Indeed,
\begin{theorem}\label{t:metrize-Fell}
     Let $\A=\overline{\cup_{n \in \N}\A_n}^{\|\cdot\|_\A}$ be an AF algebra, where we set $\mathcal{I}=(\A_n)_{n \in \N}$. For every $I,J \in \Ideals (\A)$, it holds
     \[
     d_{\beta, \mathcal{I}}(I,J) \leq \frac{1}{2}\cdot d_{\varphi, \mathcal{I}}(I,J).
     \]
    Moreover, the topologies induced by $ d_{\varphi, \mathcal{I}}$ and  $d_{\beta, \mathcal{I}}$ are equal, and thus $d_{\beta, \mathcal{I}}$ metrizes the Fell topology.
\end{theorem}
\begin{proof}
Let $I,J \in \Ideals (\A)$  be distinct. Set $N=\min \{m \in \N: I\cap \A_m \neq J \cap \A_m\}$. Hence, $ d_{\varphi, \mathcal{I}}(I,J)=2^{-N}$, and for every $m<N,$ we have $\mathring{m}(I,J)=\emptyset$. Using similar arguments as in Expression \eqref{eq:sum}, we have 
\begin{align*}
    d_{\beta, \mathcal{I}}(I,J) & = \sum_{n=N}^{\infty} \sum_{k \in \mathring{n}(I,J)} 2^{-[n+k]} \leq \sum_{n=N}^{\infty} \sum_{k=1}^\infty 2^{-[n+k]} \\
    & = \sum_{n=N}^{\infty} 2^{-n} \sum_{k=1}^\infty 2^{-k} = \sum_{n=N}^{\infty} 2^{-n}\cdot 1 = \frac{1}{2} 2^{-N} = \frac{1}{2} d_{\varphi, \mathcal{I}}(I,J)
\end{align*}

Therefore, convergence in $d_{\varphi, \mathcal{I}} $ implies convergence in $d_{\beta, \mathcal{I}} $. In particular, the topology induced by $d_{\beta, \mathcal{I}} $ is contained in the topology induced by $d_{\varphi, \mathcal{I}} $ as they are metrizable. Since both topologies are Hausdorff and $d_{\varphi, \mathcal{I}} $ is compact by \cite[Theorem 3.22]{Aguilar16c}, the topologies are equal since compact Hausdorff spaces are minimally Hausdorff. Since we know the metric $d_{\varphi, \mathcal{I}}$  metrizes the Fell topology by \cite[Theorem 3.22]{Aguilar16c}, we have $d_{\beta, \mathcal{I}} $ metrizes the Fell topology.
\end{proof}

\section{A comparison of the Hausdorff distance, $d_{\beta, \mathcal{I}}$, and $d_{\varphi, \mathcal{I}}$}\label{sct:section-example}

By construction, the metric $d_{\beta, \mathcal{I}}$ should capture more  of the  structure of the ideals than $d_{\varphi, \mathcal{I}}$. We show this explicitly in this section by comparing  our metric  with a natural metric on the ideal space that metrizes the Fell topology built by the Hausdorff distance in the case of $C(X)$. We will see in the case of $C(\Nbar)$ that $d_{\beta, \mathcal{I}}$ behaves more like this metric induced by the Hausdorff distance than $d_{\varphi, \mathcal{I}}$ for certain ideals. We begin by showing how one can place a metric on the ideal space of $C(X)$ that metrizes the Fell topology using the Hausdorff distance on $X$ in a canonical way.

Let $(X,d)$ be a compact metric space.  The Hausdorff distance on the closed subsets of $X$ is defined for every $Y,Z \in \Cl(X)$ by 
\begin{align*}
d_{H,d}(Y,Z) \coloneqq \max \left\{ \sup_{x\in Y} d(x,Z), \, \sup_{x\in Z}d(Y,x) \right\} 
,\end{align*}
where $d(a,B) = \inf_{x\in B} d(a,x)$ is the distance from a point $a \in X$ to the subset $B \subseteq X$, is a metric on $\Cl(X)$ (see \cite[Chapter 7]{Burago01}). Using Example \ref{e:comm}, we can use this to place a metric on $\Ideals(C(X))$ defined  for each $I,J \in \Ideals(C(X))$  by
\[
d_{H,d,I} (I,J)\coloneqq d_{H,d}(\Phi_I^{-1}(I), \Phi^{-1}(J)).
\]
The following result might be well-known, but we state it here for convenience.
\begin{theorem}\label{t:haus-ideal}
    Let $(X, d)$ be a compact metric space. The topology induced by $d_{H,d_1,I}$ on $\Ideals(C(X))$ is the Fell topology.
\end{theorem}
\begin{proof}
    The primitive ideals of $C(X)$ are of the form $I_{\{x\}}=\{f \in C(X): f(x)=0\}$ for all $x \in X$,   and the set of primitive ideals $\{I_{\{x\}} \subseteq C(X) : x \in X\}$ equipped with the Jacobson topology  is homeomorphic to $(X,d)$ (this is a well-known result, which follows from \cite[Proposition 4.3.3]{Pedersen79}). Finally, the Fell topology of Definition \ref{topological-Fell-def} on the closed subsets of $(X,d)$ is homeomorphic to the topology on the closed subsets of $(X,d)$ induced by the Hausdorff distance by \cite[Corollary 5.1.11]{Beer}. Hence, by Definition \ref{Fell-topology}, the proof is complete.
\end{proof}

The next few results display that $d_{\beta, \mathcal{I}}$ is capable of seeing more of the ideal structure than $d_{\varphi,\mathcal{I}}$.

\begin{proposition}\label{p:haus-cnbar}
    Consider $C(\Nbar)$. Let $m,n,k \in \N$ and set $A=\{2^{-m}\}\in Cl( \Nbar)$ and $B=\{2^{-n}, 2^{-[n+k]}\}\in Cl(\Nbar)$. It holds that 
    \[
   d_{H,d_1,I}(I(A), I(B))=\begin{cases}
        \left|2^{-m}-2^{-[n+k]} \right| & m \leq n
        \\
\left| 2^{-m} - 2^{-n} \right| &   m > n,
    \end{cases}
    \]
    where $d_1$ is the absolute value metric on $\Nbar.$
\end{proposition}
\begin{proof}
    If $m \leq n$, then  $2^{-m} \geq 2^{-n}$ and $|2^{-m} - 2^{-[n+k]}| \geq |2^{-m} - 2^{-n}|$ since $2^{-n}\geq 2^{-[n+k]}$. Thus \[d_{H,d_1}(A,B)=|2^{-m} - 2^{-[n+k]}|.\]
    
    Next,  if $m > n$, then $2^{-m} < 2^{-n}$, and similarly, $|2^{-m} - 2^{-[n+k]}| < |2^{-m} - 2^{-n}|$. Thus, \[d_{H,d_1}(A,B)=|2^{-m} - 2^{-n}|.\]
\end{proof}
We proceed by calculating these distances in $d_{\varphi, \mathcal{I}}$.

\begin{proposition}\label{p:dphi-nbar}
     Consider $C(\Nbar)$ and $\mathcal{I}=(C_p)_{p \in \N}$. Let $m,n,k \in \N$ and set $A=\{2^{-m}\}\in Cl( \Nbar)$ and $B=\{2^{-n}, 2^{-[n+k]}\}\in Cl(\Nbar)$.  It holds that
    \[ d_{\varphi,\mathcal{I}}(I(A),I(B)) = 2^{-(\min \{m,n\} +1 )}.\]
\end{proposition}
\begin{proof}
    Let $m \leq  n$. Recall $I(A)$ is the set of functions in $C(\Nbar)$ that vanish at $2^{-m}$; and $I(B)$ is the set of functions that vanish on $\{2^{-n}, 2^{-[n+k]}\}$. Since $C_p$ for each $p \in \N$ denotes the subalgebra of functions that are eventually constant at $2^{1-p}$, then $I(A)\cap C_p=I(B)\cap C_p$ for $p \leq m$. However, $I(A) \cap C_{m+1}\neq I(B)\cap C_{m+1}$ which implies
    \[
    d_{\varphi,\mathcal{I}}(I(A),I(B))=2^{-(m+1)}.
    \]

    Next, assume that $m >n$. Then mutatis mutandis, we get
    \[
     d_{\varphi,\mathcal{I}}(I(A),I(B))=2^{-(n+1)},
    \]
    as desired.
\end{proof}
We note the above proof already begins to show $d_{\varphi, \mathcal{I}}$ ignores part of the ideal structure as the ideal associated to $2^{-[n+k]}$ is ignored. In the following proposition, we show our new metric does capture this part of the ideal. We focus on the same case as above since this is all we need to show our metric preserves more information.

\begin{proposition}
Consider $C(\Nbar)$. Let $m,n,k \in \N$ such that $m<n$ and set $A=\{2^{-m}\}\in Cl( \Nbar)$ and $B=\{2^{-n}, 2^{-[n+k]}\}\in Cl(\Nbar)$. It holds that 
\begin{align*}
d_{\beta, \mathcal{I}}(I(A), I(B))=\frac{1}{2}\left( 2^{-2(n+k)}+2^{-m}+4^{-n}\right).
\end{align*}
\end{proposition}
\begin{proof}
Again, recall the set of functions in $C(\Nbar)$ that vanish at $2^{-m}$ is denoted $I(A)$; the set of functions which vanish on $\{2^{-n}, 2^{-[n+k]}\}$ is denoted $I(B)$; and $C_p$ for each $p \in \N$ is the subalgebra of functions that are eventually constant at $2^{1-p}$. 

Let $p \in \N$. We have by Expression \eqref{eq:nbar} and Theorem \ref{t:ideal-af}, that 
\[
\mathring{p}(I(A))=\begin{cases}
    \{1,2,\ldots, p\} & p \leq m,\\
     \{1,2,\ldots, p-1\} & p=m+1,\\
     \{1, 2, \ldots, m,m+2,\ldots, p\} & p>m+1,
\end{cases}
\]
and
\[
\mathring{p}(I(B))=\begin{cases}
    \{1,2,\ldots, p\} & p \leq n\\
      \{1,2,\ldots, p-1\}                  & p=n+1\\
      \{1,2,\ldots, n,n+2, \ldots,p \} & n+1<p\leq n+k\\
      \{1,2,\ldots, n,n+2, \ldots, p-1\} & p=n+k+1\\
      \{1,2,\ldots, n,n+2, \ldots, n+k, n+k+2, \ldots, p\} & p>n+k+1.
\end{cases}
\]

First, assume $m<n$. Let $p \in \N$. Then
\[
\mathring{p}(I(A), I(B))=\begin{cases}
 \emptyset    & p \leq m,\\
 \{m+1\} & m+1\leq p\leq n,\\
  \{m+1,n+1\}& n+1\leq p\leq n+k\\
  \{m+1,n+1, n+k+1\} &  p\geq n+k+1,
\end{cases}
\]
Hence, 
\begin{align*}
    d_{\beta, \mathcal{I}}(I(A), I(B))& =\sum_{p=1}^\infty \bigg( \sum_{r \in \mathring{p}(I(A), I(B))} 2^{-[p+k]} \bigg) = \sum_{p=m+1}^\infty \bigg( \sum_{r \in \mathring{p}(I(A), I(B))} 2^{-[p+k]}\bigg)\\
    & =  \sum_{p=m+1}^n 2^{-[p+(m+1)]} + \sum_{p=n+1}^{n+k} \big(2^{-[p+(m+1)]}+2^{-[p+(n+1)]}\big)\\
    & \quad \quad + \sum_{p=n+k+1}^\infty  \big(2^{-[p+(m+1)]}+2^{-[p+(n+1)]}+2^{-[p+(n+k+1)]}\big)\\
    & = \big(2^{-m-1}-2^{-m-n-1}\big) + (2^k-1)\big(2^{-k-2n-1}+2^{-k-m-n-1}\big)\\
    & \quad \quad + 2^{-2k-2n-1}+2^{-k-2n-1}+2^{-k-m-n-1}\\
    & = \frac{1}{2}\big( 2^{-2(n+k)}+2^{-m}+4^{-n}\big).\qedhere
\end{align*}
\end{proof}

Now, consider the case when $m=1, n=2$, and $k=1$, then
\begin{align*}
d_M(I(A), I(B)) = \begin{cases}
3/8, & \qquad M=H,d_1,I \\
1/4 & \qquad M=\varphi, \mathcal{I} \\
37/128 & \qquad M= \beta ,\mathcal{I} \\
\end{cases}
\end{align*}
But for the case $m=1, n=2,$ and $k=2$, we have
\begin{align*}
d_M(I(A), I(B)) = \begin{cases}
7/16, & \qquad M=H,d_1,I \\
1/4 & \qquad M=\varphi, \mathcal{I} \\
145/512 & \qquad M= \beta ,\mathcal{I}\\
\end{cases}
\end{align*}

Thus, we see very explicitly that $d_{\beta, \mathcal{I}} $ can see more of the ideal structure than $    d_{\varphi, \mathcal{I}}$ in this case just as this canonically induced dual  Hausdorff distance. Of course, this could have been guessed from the construction of our new metric, but we felt it important to display an explicit example. Now, one may wonder if the Hausdorff distance is maybe better to work with since the proof of the calculation in $d_{\beta, \mathcal{I}}$ was longer, but there are two reasons why this is not true. First, once one considers infinite sets, dealing with infimum and supremum in the Hausdorff distance becomes much more troublesome to work with, whereas the calculations in our new metric still only deal with terms from geometric series. Second, the Hausdorff distance only induces a metric on the ideals in the commutative case, whereas our new metric provides many more possibilities for more precise measurement in the noncommutative setting as well.

%========================================================================
%=================BIBLIOGRAPHY===========================================
%========================================================================

\bibliographystyle{amsplain}
\bibliography{thesis}

\end{document}